\documentclass[tikz]{nmd/article}
\ifnmd
\fi

\usepackage{upquote}

\nmdnewtheorem{Strong Neuwirth Conjecture}
\nmdnewtheorem{Neuwirth Conjecture}
\nmdnewtheorem{Even Boundary Slope Conjecture}
\newcommand{\monodromy}{(\sigma_1 \sigma_2 \sigma_3 \sigma_4)^{13} \sigma_1 \sigma_4 \sigma_3 \sigma_2}

\title{A knot without a nonorientable \\ essential spanning surface} 

\author{Nathan M. Dunfield}
\givenname{Nathan}
\surname{Dunfield}
\address{ Dept.~of Math., MC-382 \\
          University of Illinois \\
          1409 W. Green St. \\
          Urbana, IL 61801 \\ 
          USA
}
\email{nathan@dunfield.info}
\urladdr{http://dunfield.info}

\arxivreference{}
\arxivpassword{}

\begin{document}

\begin{abstract} 
  This note gives the first example of a hyperbolic knot in the
  \3-sphere that lacks a nonorientable essential spanning surface;
  this disproves the Strong Neuwirth Conjecture formulated by Ozawa
  and Rubinstein.  Moreover, this knot has no even strict boundary
  slopes, disproving the Even Boundary Slope Conjecture of the same
  authors.  The proof is a rigorous calculation using Thurston's
  spun-normal surfaces in the spirit of Haken's original normal surface
  algorithms.
\end{abstract}
\maketitle

\section{Introduction}

Let me start with the definitions needed to precisely state the first
result outlined in the abstract.  Throughout, all \3-manifolds will be
orientable.  A properly embedded orientable surface $S \subset M^3$ is
\emph{essential} if it is incompressible, $\partial$-incompressible,
and not isotopic (rel boundary) into $\partial M$.  A nonorientable
$S$ is defined to be essential when the boundary of a regular
neighborhood $N(S)$ is essential.  For a tame knot $K$ in $S^3$,
consider its exterior $E(K) = S^3 \setminus \mathring{N}(K)$, which is
a compact manifold with torus boundary.  Any such $K$ has a Seifert
surface, that is, there is an embedded surface $S$ in $S^3$ with
$\partial S = K$.  Moreover, there is always a Seifert surface whose
intersection with $E(K)$ is essential, e.g.~any Seifert surface of
minimal genus.  Additionally, any $K$ bounds a \emph{nonorientable}
spanning surface (just add a small half-twisted band to the boundary
of a Seifert surface).  Ichihara, Ohtouge, and Teragaito studied
nonorientable spanning surfaces and asked whether there is always such a
surface that is essential in $E(K)$
\cite{IchiharaOhtougeTeragaito2002}.  While torus knots $T_{p,q}$ with
$p$ and $q$ both odd lack nonorientable essential spanning surfaces
\cite[Example 3.11]{OzawaRubinstein2012}, Ozawa and Rubinstein
\cite{OzawaRubinstein2012} posited that these are the only such
examples:
\begin{Strong Neuwirth Conjecture}[\cite{OzawaRubinstein2012}]
\label{conj:strong}
  Every prime non-torus knot in $S^3$ has a nonorientable essential
  spanning surface $K$.
\end{Strong Neuwirth Conjecture}
As the name suggests, this conjecture implies the Neuwirth Conjecture
from 1963 \cite[Conjecture B]{Neuwirth1963}, which remains open:
\begin{Neuwirth Conjecture}
\label{conj:neuwirth}
  Every non-trivial knot $K$ in $S^3$ lies on a closed surface $F$ in
  $S^3$ where $K$ is nonseparating in $F$ and $F \cap E(K)$ is
  essential.
\end{Neuwirth Conjecture}
Part of the motivation in \cite{OzawaRubinstein2012} for formulating
Conjecture~\ref{conj:strong} is that in almost all cases where
Conjecture~\ref{conj:neuwirth} is known it is by proving this stronger
statement; please see \cite{OzawaRubinstein2012} for details and an
overview of work in this direction. My main result here disproves
Conjecture~\ref{conj:strong}, cutting off this approach to proving the
Conjecture~\ref{conj:neuwirth} in full generality:
\begin{theorem}\label{thm:main}
  Let $K$ in $S^3$ be the braid closure of $\monodromy$.  Then $K$ is
  a hyperbolic knot without a nonorientable essential spanning
  surface.
\end{theorem}
The knot $K$ was introduced in \cite{CallahanDeanWeeks1999} as
$k6_{36}$ where they described it as the twisted torus knot
$T(5,17)_{4,-1}$. While its diagram in Figure~\ref{fig:knot} has some
56 crossings, its exterior is not complicated: it is the hyperbolic
\3-manifold $s800$ from \cite{CallahanHildebrandWeeks1999} which can
be triangulated with $6$ ideal tetrahedra and has volume about
$5.34821999$.

\begin{figure}
\begin{center}
\begin{tikzpicture}[scale=0.5, rotate=195, color=black, line width=2]
\foreach \angle in {0,24,...,360}{
  \begin{scope}[rotate=\angle]
    \begin{scope}
      \draw (6:4) .. controls (14:4.05) and (24:5.25) .. (48:6);
      \draw (48:6) .. controls (60:7) and (84:8.2) .. (92:8);
    \end{scope}
  \end{scope}
};
\foreach \angle in {0,24,...,360}{
  \begin{scope}[rotate=\angle]
    \draw[color=white, line width=7.5pt]
                    (30:4) .. controls (12:4) and (12:8.03) .. (-4:8.005);
    \draw (30.3:4) .. controls (12:3.95) and (12:8.03) .. (-4.15:8.005);
    \end{scope}
};
\fill[color=white] 
   (12:3) -- (-4:8.4) --(22:8.6) -- (40:8.5) -- (70:3) -- cycle;
\draw (5.1:3.99) .. controls (17.9:3.98) and (32:6.5) .. (43.5:6.93);
\draw (-2.5:7.03) .. controls (15.6:8.4) and (40:7.9) .. (41.2:7.985);
\foreach \color/\width/\cap in {white/7pt/butt, black/2pt/round}{
  \begin{scope}[color=\color, line width=\width, line cap=\cap]
    \draw (2.2:4.815) .. controls (17.5:5.95) and (31:3.8) .. (58:4.07);
    \draw (-4.5:8.005)  .. controls (17.0:8.08) and (35:5.4) ..  (47:5.94);
    \draw (-0.7:5.96) .. controls (17:7.3) and (27:4.40) .. (51.5:4.88);
  \end{scope}
};
\draw[color=black] (4.43, 1.96) -- (4.52, 2.17);
\end{tikzpicture}
\end{center}
\caption{The knot $K = k6_{36}$ is the braid closure of $\monodromy$
  as well as the twisted torus knot $T(5,17)_{4,-1}$
  \cite{CallahanDeanWeeks1999}.}
\label{fig:knot}
\end{figure}

Theorem~\ref{thm:main} is an immediate corollary of a more technical
result for which I need more definitions.  Any essential $S$ in $E(K)$
with nonempty boundary has a \emph{boundary slope}, namely the common
unoriented isotopy class of the components of $\partial S$ in the
torus $\partial E(K)$; as usual, boundary slopes are recorded as
elements of $\Q \cup \{\infty\}$ using the standard homological
framing on $\partial E(K)$.  In our context, the boundary slope of an
essential surface $S$ is \emph{strict} if $S$ is not a Seifert surface
corresponding to a fibration of $E(K)$ over the circle.  I will show:
\begin{theorem}\label{thm:tech}
  Let $K \subset S^3$ be as in Theorem~\ref{thm:main}.  Then $E(K)$
  has strict boundary slopes exactly $\{-77, -71, -211/3, -69, -67\}$.
\end{theorem}
Theorem~\ref{thm:tech} implies Theorem~\ref{thm:main}
because the boundary slope of a nonorientable essential spanning
surface $S$ must be an even integer: it is integral because $S$
intersects a meridian curve exactly once, and it is even as $S$
demonstrates that the boundary of $S \cap E(K)$ is zero in
$H_1(E(K); \F_2)$.  As promised in the abstract,
Theorem~\ref{thm:tech} disproves:
\begin{Even Boundary Slope Conjecture}[\cite{OzawaRubinstein2012}]
  \label{conj:even}
  For any prime non-torus knot $K$, there is an essential surface
  $E(K)$, not a Seifert surface, whose boundary slope is a
  rational number with even numerator.
\end{Even Boundary Slope Conjecture}

\section{Proof}

The proof of Theorem~\ref{thm:tech} will be a straightforward rigorous
computation using Thurston's theory of spun-normal surfaces, which is
a version of Haken's normal surface theory tuned to the
setting of ideal triangulations of cusped manifolds (see
\cite{Tillmann2008} or \cite{DunfieldGaroufalidis2012} for general
background).  Let $\cT$ be the standard 6-tetrahedra ideal
triangulation of the exterior $E(K)$ of $K = k6_{36}$ used in
\cite{SnapPy}, which is the one given in
\cite{CallahanHildebrandWeeks1999} with its peripheral framing changed
to the homologically natural one for the complement of a knot in
$S^3$.

\begin{lemma}\label{lem:normal}
  The set of strict boundary slopes for $E(K)$ is contained in the set
  of boundary slopes of spun-normal surfaces in $\cT$, which
  is
  \[
  \{-77, -71, -211/3, -69, -67\}.
  \]
\end{lemma}
Before proving this, let me point out that Lemma~\ref{lem:normal} is
already enough to establish Theorem~\ref{thm:main}. Also, you might be
troubled by the absence of $0$ on the above list of slopes; however, as $K$
is the closure of a positive braid, the manifold $E(K)$ fibers over the
circle and so $0$ need not be a \emph{strict} boundary slope.

\begin{proof}[Proof of Lemma~\ref{lem:normal}]
  Let $S$ be any essential surface in $E(K)$ which is not a fiber.  We
  will assume that $S$ is orientable, since if not we can replace it
  with the boundary of its regular neighborhood, which is an
  orientable essential surface with the same boundary slope.  By
  \cite{HoffmanEtAl2016}, there is a geometric solution to the gluing
  and completeness equations for $\cT$, that is, one where all
  tetrahedra are positively oriented.  Thus the interior of $E(K)$ is
  hyperbolic and no edge of $\cT$ is homotopically peripheral.
  Therefore, by \cite[Theorem 1.6]{Walsh2011}, the surface
  $S$ can be isotoped into spun-normal form with respect to $\cT$.
  (Technical note: the hypotheses in \cite{Walsh2011} require that $S$
  is not a \emph{virtual} fiber, but the only virtual fibers in the
  exterior of a knot in $S^3$ are actual fibers.)  Thus the boundary
  slope of $S$ is also the boundary slope of a spun-normal surface,
  proving the first part of the lemma.  To see that the spun-normal
  surfaces have only the boundary slopes listed above, one simply
  computes the boundary slopes of the finite collection of vertex
  spun-normal surfaces.  This is easily done rigorously using SnapPy
  \cite{SnapPy} or Regina \cite{Regina}; for example, in the former
  one simply does: \verb|Manifold('K6_36').normal_boundary_slopes()|
\end{proof}
\noindent
As mentioned, Lemma~\ref{lem:normal} immediately proves
Theorem~\ref{thm:main}.  The stronger statement of
Theorem~\ref{thm:tech} now follows by combining Lemma~\ref{lem:normal}
with:
\begin{lemma}\label{lem:tech}
  The exterior $E(K)$ contains no closed essential surfaces, and the
  Dehn fillings of $E(K)$ along $\{-77, -71, -211/3, -69, -67\}$ each
  yield Haken manifolds.
\end{lemma}
I established Lemma~\ref{lem:tech} using the breakthrough work of
\cite{BurtonOzlen2014, BurtonCowardTillmann2013} as implemented in
\cite{Regina}.  The complete script I used for this can be found at
\cite{ancillary} and the total running time was less than 20 seconds;
since Lemma~\ref{lem:tech} is not needed to prove
Theorem~\ref{thm:main}, I simply refer you to the code for details.

\subsection{Remark}
Given the original motivation for Conjecture~\ref{conj:strong}, you
might wonder whether $K$ satisfies Conjecture~\ref{conj:neuwirth},
namely that $K$ lies on a closed surface $F$ in $S^3$ where $K$ is
nonseparating in $F$ and $F \cap E(K)$ is essential.  In fact it does,
and here is one way to see this.  SnapPy finds a vertex spun-normal
surface $S$ in $\cT$ which has exactly two boundary components, each
of slope $-77$, and which is essential by \cite[Theorem
1.1]{DunfieldGaroufalidis2012}.  As the boundary slope of $S$ is an
integer, we can piece together the two boundary components of $S$ to
get a surface $F$ in $S^3$ on which $K$ lies.  The knot $K$ does not
separate $F$ because, as a vertex surface, the original surface $S$ is
connected.

\subsection{Further examples}  

The knot $K$ was found by a computer search through all 502 knots
whose complements can be triangulated with 8 or fewer tetrahedra
\cite{CallahanDeanWeeks1999, ChampanerkarKofmanPatterson2004,
  ChampanerkarKofmanMullen2014}. Other examples which lack
nonorientable essential spanning surfaces include $k5_{17}, k7_{29},
k7_{64}$ and $k8_{114}$.  The complete search took less than 10
seconds to run.

\section{Acknowledgments} 

This work was partially supported by US NSF grant DMS-1510204 and
partially conducted at the Institute for Advanced Study with the
support of the Giorgio and Elena Petronio Fellowship Fund and the Bell
Companies Fellowship Fund.  I thank the lunch table crowd at IAS for
comments and encouragement, as well as the referee for their
comments. 

{\RaggedRight 
\small
\bibliographystyle{nmd/math} 
\bibliography{\jobname}
}
\end{document}